\documentclass[10pt,final]{amsart}

\RequirePackage[strict]{csquotes}
\RequirePackage[initials,non-compressed-cites]{amsrefs}

\usepackage{upref}

\usepackage{amsmath,amsthm,amscd,amssymb,siunitx}
\usepackage{braket,mleftright}
\DeclareMathAlphabet\mathrsfso{U}{rsfso}{m}{n}

\newcounter{smalllist}
\newenvironment{SL}{\begin{list}{{\rm\roman{smalllist})}}{%
\setlength{\topsep}{0mm}\setlength{\parsep}{0mm}\setlength{\itemsep}{0mm}%
\setlength{\labelwidth}{2em}\setlength{\leftmargin}{2em}\usecounter{smalllist}%
}}{\end{list}}

\usepackage{ragged2e}
\justifying
\fussy

\newcommand{\ie}{\textit{i.e.}\;}
\newcommand{\eg}{\textit{e.g.}\;}
\newcommand{\cf}{\textit{cf.}\;}

\newcommand{\bbC}{\mathbb{C}}
\newcommand{\bbR}{\mathbb{R}}

\newcommand{\what}[1]{\widehat{#1}}		
\newcommand{\mrm}[1]{\mathrm{#1}}		
\newcommand{\ol}[1]{\overline{#1}}		
\newcommand{\co}{\colon}				
\renewcommand{\vrt}{\,\vert\,}			
\newcommand{\lto}{\rightarrow}			
\newcommand{\abs}[1]{\lvert#1\rvert}	
\newcommand{\norm}[1]{\lVert#1\rVert}	
\newcommand{\img}{\mrm{i}}				
\newcommand{\op}{\oplus}				
\newcommand{\om}{\ominus}				
\newcommand{\setm}{\smallsetminus}		
\newcommand{\hsum}{\,\what{+}\,}		
\newcommand{\hop}{\,\what{\op}\,}		

\DeclareMathOperator{\im}{Im}     
\DeclareMathOperator{\re}{Re}     
\DeclareMathOperator{\Ind}{Ind}   
\DeclareMathOperator{\Ker}{Ker}   
\DeclareMathOperator{\Lin}{Lin}   
\DeclareMathOperator{\CLin}{\textit{C}\,Lin}   
\DeclareMathOperator{\St}{St}     

\newcommand{\msH}{\mathrsfso{H}}
\newcommand{\msD}{\mathrsfso{D}}
\newcommand{\msR}{\mathrsfso{R}}
\newcommand{\msG}{\mathrsfso{G}}
\newcommand{\msN}{\mathrsfso{N}}
\newcommand{\msL}{\mathrsfso{L}}
\newcommand{\msB}{\mathrsfso{B}}
\newcommand{\msK}{\mathrsfso{K}}
\newcommand{\msA}{\mathrsfso{A}}
\newcommand{\msU}{\mathrsfso{U}}

\newcommand{\hx}{\hat{x}}
\newcommand{\hy}{\hat{y}}
\newcommand{\hl}{\hat{l}}
\newcommand{\hh}{\hat{h}}

\newcommand{\GB}{\Gamma^{\scriptscriptstyle{B} } }
\newcommand{\GBp}{\Gamma^{\scriptscriptstyle{B^\prime} } }
\newcommand{\GA}{\Gamma^{\scriptscriptstyle{A} } }
\newcommand{\GAp}{\Gamma^{\scriptscriptstyle{A^\prime} } }
\newcommand{\phB}{\phi^{\scriptscriptstyle{B} } }

\newcommand{\hGm}{\hat{\Gamma}}
\newcommand{\mGm}{\dot\Gamma}
\newcommand{\mT}{\dot T}
\newcommand{\mN}{\dot N}

\newcommand{\tU}{\tilde{U}}
\newcommand{\tK}{\tilde{K}}

\newtheorem{thm}{Theorem}[section]
\newtheorem{lem}[thm]{Lemma}
\newtheorem{prop}[thm]{Proposition}
\newtheorem{cor}[thm]{Corollary}
\theoremstyle{remark}
\newtheorem{rem}[thm]{Remark}

\newtheorem{exam}[thm]{Example}

\theoremstyle{definition}
\newtheorem{defn}[thm]{Definition}

\begin{document}
\title[]{On the similarity of boundary triples
for dual pairs}
\author{Rytis Jur\v{s}\.{e}nas}
\address{Vilnius University,
Institute of Theoretical Physics and Astronomy,
Saul\.{e}tekio ave.~3, 10257 Vilnius, Lithuania}
\email{rytis.jursenas@tfai.vu.lt}
\thanks{}
\date{\today}
\begin{abstract}
The Weyl family of a dual pair $A\subseteq B^c$
of operators in a Krein space determines
a minimal boundary triple uniquely up to similarity;
if $A=B$, a necessary and sufficient
condition in order that the similarity should be
unitary is given.
\end{abstract}
\maketitle
\section{Introduction}
Let $A$, $B$ be closed linear relations in a
complex separable Krein
space $\msH$. The pair $(A,B)$ is a dual pair (dp)
if $A\subseteq B^c$
(the adjoint, \cite[Definition~2.6.7]{Azizov89});
in \cite{Behrndt23,Schmudgen20} a dp is an adjoint pair
of closed densely defined operators in a Hilbert space.
Adapted from \cite{Hassi13} to the Krein space setting
is a boundary triple (bt)
$\Pi=(\msG_0\op\msG_1,\GB,\GA)$;
as $\GA$ ($\GB$) is uniquely determined by
$\GB$ ($\GA$),
$\Pi\equiv(\msG,\GB)$, $\msG=\msG_0\op\msG_1$.
A bt for a dp exists always, see
\cite{Malamud02}, also Remark~\ref{rem:bt}.
To avoid repetition
we refer instead to \cite{Behrndt23,Brown19,Hassi13,Malamud10}
and a list of sources
therein for motive and applications.

The problem we address here is whether
the Weyl family
$\lambda\mapsto M_{\GB}(\lambda)=\GB(\lambda I)$
on a suitable subset of the complex plane determines
$\Pi$ uniquely. If $A=B$ has equal defect numbers
and $\Pi$ is
an ordinary boundary triple (obt),
$\Pi_\Gamma=(\msG,\Gamma)$, the Hilbert space
solution is found essentially in
\cite{Derkach95a,Derkach91,Langer77}:
The Weyl family
$\lambda\mapsto M_{\Gamma}(\lambda)=\Gamma(\lambda I)$
on $\bbC_*=\bbC\setm\bbR$
determines a minimal obt $\Pi_\Gamma$ uniquely
up to unitary equivalence. Related statements in
a Pontryagin space exist too,
\cite{Jursenas23,Hassi16,Hassi98}.

For dp's
it is known that $M_{\GB}$ determines
$\Pi$ up to so-called weak
similarity\footnote{The original proof in
M.M. Malamud, V.I. Mogilevskii, On Weyl functions and
$Q$-functions of dual pairs of linear
relations, Dop. Akad. Nauk Ukr. (1999), no.4, 32–37,
is inaccessible to us; but
see \cite{Hassi05}
for the definition of weak similarity, which
is due to \cite{Arov79}.
The sources, however, are not critical
for our presentation.}. We prove that
the weak similarity can be replaced by the similarity
in its usual sense. Although now the spectral properties
are preserved, it is nonetheless of general interest
to explore whether the similarity can be chosen unitary,
just like for obt's.
If $A=B$, this is indeed so iff the boundary
operators $\GB$ and $\GBp$ in bt's
$\Pi=(\msG,\GB)$ and $\Pi^\prime=(\msG,\GBp)$
for dp's $(A,A)$ and $(A^\prime,A^\prime)$ in
Krein spaces $\msH$ and $\msH^\prime$
satisfy the condition:
$(\GB)^{-1}\GBp$ is a
unitary relation $\msH^\prime_\Gamma\lto\msH_\Gamma$;
the Krein space
$(\msH_\Gamma,[\,\,\cdot\,\,,\,\,\cdot\,\,]_\Gamma)$
($(\msH^\prime_\Gamma,
[\,\,\cdot\,\,,\,\,\cdot\,\,]^\prime_\Gamma)$) is as defined in
\cite[Section~2.1]{Azizov89}. An equivalent characterization
is that $(\GB)^{-1}\GBp=(\GA)^{-1}\GAp$,
see Sections~\ref{sec:mobt}--\ref{sec:bt} for details.

As an exercise, the
results on unitarity in
\cite{Hassi13,Arlinskii05}
fall within the general criterion.
As an application, in Section~\ref{sec:lft}
we answer the question raised in
\cite[Remark~3.16]{Hassi13} and
in Section~\ref{sec:Dbt} we show that
the Weyl function of a simple $\pi$-symmetric
operator (\cite{Azizov03})
of class $(L)$ (\cite{Langer75})
determines a $D$-boundary triple up to
unitary equivalence; a $D$-bt of a symmetric
operator in a $\Pi_0$-space
was originally defined in \cite{Mogilevski06}, see also
\cite{Mogilevski11,Mogilevski09}.

The main idea in our approach is to
associate with a dp $(A,B)$ a block diagonal
linear relation $T=\bigl(\begin{smallmatrix}A & 0 \\
0 & B \end{smallmatrix}\bigr)$, which is
closed symmetric in a suitable Krein space.
Although an off-diagonal analogue is also used
\eg in \cite{Schmudgen20}, \cf \cite{Mogilevskii20},
here we go on a step further
in determining the correspondence between
a bt $\Pi$ for $(A,B)$
and an obt $\Pi_\Gamma$ for
$T^c=\bigl(\begin{smallmatrix}B^c & 0 \\
0 & A^c \end{smallmatrix}\bigr)$.
As a matter of fact our method works for general
isometric boundary pairs.
\section{Minimal ordinary boundary triple}\label{sec:mobt}
Standard sources for the Hilbert space (linear) relations
and the Krein ($J$-)spaces are
\cite{Behrndt20,Hassi09,Azizov89}.
Assuming that the reader is familiar with
the theory of boundary value spaces,
well established definitions are not detailed.

Notation is as in \cite{Azizov89}. Particularly,
if $T$ is a relation in a Krein space
$(\msH,[\,\,\cdot\,\,,\,\,\cdot\,\,])$,
the points of regular type
$(\lambda\in r(T))$ as well as other
spectral points are defined as in
\cite[Section~2.6]{Azizov89}. We use
$r_0(T)$ for the set of $\lambda\in\bbC$
such that $\msR_{T-\xi I}$ is a subspace
(closed lineal) in $\msH$
for both $\xi=\lambda$ and $\xi=\bar{\lambda}$.
The kernel
$\Ker_\lambda T=\Ker(T-\lambda I)$.

Let $T$ be a closed symmetric relation in
$\msH$ with equal nonzero defect numbers.
Let $\Pi_\Gamma=(\msG,\Gamma)$ be an obt for $T^c$;
\ie $\Gamma$
is a unitary surjective
operator $\msH_\Gamma\lto\msG_\Gamma$,
$\msD_{\Gamma}=T^c$.
As usual
$\Gamma_i=\pi_i\Gamma$ for $i=0,1$
(with $\pi_0\co\msG_\Gamma\lto\msG$, $\binom{l}{h}\mapsto l$
and $\pi_1\co\msG_\Gamma\lto\msG$, $\binom{l}{h}\mapsto h$),
$T_i=\Ker\Gamma_i$, and the Weyl family
$\lambda\mapsto M_\Gamma(\lambda)=\Gamma(\lambda I)$
on $\bbC$, where
the domain restriction $\Gamma\vrt\lambda I=
\Gamma\vrt\lambda I\cap\msD_\Gamma$, with similar
interpretation throughout.
Let
\[
r(T,T_0)=r(T_0)\cap r_0(T)\cap
\sigma_p(T^c)\cap\sigma_r(T)\,.
\]
\begin{rem}
Let $\lambda\in r(T_0)\cap r_0(T)$.
A real
$\lambda\in\sigma_p(T^c)$.
For a non-real
$\lambda\in\sigma_p(T^c)$
(resp. $\lambda\in\sigma_r(T)$ or
equivalently $\bar{\lambda}\in\sigma_p(T^c)$)
it suffices that
$\Ker_\lambda JT^c\notin\msA^-$
(resp. $\Ker_{\bar{\lambda}} JT^c\notin\msA^-$)
\cite[Section~1.8]{Azizov89}, \eg that
$\Ker_\lambda JT^c$
(resp. $\Ker_{\bar{\lambda}} JT^c$) be
positive in $\msH$, as is the case if
$T\in(L)$ and $\abs{\im\lambda}$ is sufficiently
large.
In the Hilbert space context
$\lambda\in r(T,T_0)$ is just a regular
point $\lambda\in\rho(T_0)$.
\end{rem}
\begin{defn}
$\{\msU,M\}$ is a holomorphic element
if $\msU$ is an open subset in the extended complex plane and
$M$
is a closed relation
holomorphic on $\msU$.
\end{defn}
A holomorphic closed relation is as defined in
\cite{Derkach95}.
$M_\Gamma(\lambda)=M_\Gamma(\bar{\lambda})^*$
for $\lambda\in r_0(T)$
($\lambda\in r(T_0)\cap r_0(T)$)
is a closed relation in $\msG$
(a closed bounded operator from
$\msB(\Gamma_0(\lambda I),\msG)$),
see Proposition~\ref{prop:M}.
\begin{defn}
An obt $\Pi_\Gamma\equiv\{\msU,\Pi_\Gamma\}$
for $T^c$ is minimal
if $\{\msU,M_\Gamma\}$ and $\{\msU^*,M_\Gamma\}$
($\msU^*=\{\lambda\vrt\bar{\lambda}\in\msU \}$) are
holomorphic elements
for some open subset $\msU\subseteq r(T,T_0)$,
and
$\CLin\{\Ker_\lambda T^c\vrt\lambda\in
\msU\cup\msU^* \}=\msH$.
\end{defn}
\begin{rem}
Essentially we drop the assumption that
$\rho(T_0)$ ($\subseteq r(T_0)\cap r_0(T)$) is nonempty.
It can happen that $\rho(T_0)=\emptyset$---as well as
$\rho(T_1)=\emptyset$, so that
making $\Pi_\Gamma$ into the transposed obt would
not change the picture---and yet $M_\Gamma(\lambda)$
would be holomorphic for
all $\lambda\in\bbC^-=r(T,T_0)$ and then
for all $\lambda\in\bbC_*$:
This happens
if \eg (Section~\ref{sec:Dbt})
$T=\bigl(\begin{smallmatrix}A & 0 \\ 0 & A
\end{smallmatrix} \bigr)$
in $\msH_\Gamma$, where $A$ is a closed densely
defined simple symmetric operator in a Hilbert space
$\msH$, and
$T_0=\bigl(\begin{smallmatrix}A_0 & 0 \\ 0 & A^*_0
\end{smallmatrix} \bigr)$,
where $A_0\supset A$ is a closed maximal symmetric
operator with $\rho(A_0)=\bbC^+$.
On the other hand, there are
examples of obt's with $\rho(T_0)\neq\emptyset$
among $\lambda_0$-standard $T$
\cite[Section~5.2.4]{Azizov89},
\cf \cite[Corollary~3.16]{Malamud02};
note that $I+(\lambda-\bar{\lambda})(T_0-\lambda I)^{-1}$
($\lambda\in\rho(T_0)$) maps $\Ker_{\bar{\lambda}} T^c$
bijectively onto $\Ker_{\lambda}T^c$.
\end{rem}
\begin{rem}
In a $\Pi_\kappa$-space:
$M_\Gamma$ is realized minimally
\cite{Hassi16,Hassi98} if
$\CLin\{\Ker_\lambda T^c\vrt \lambda\in\rho(T_0)
\}=\Pi_\kappa$;
a boundary
operator $\Gamma$ is minimal \cite{Behrndt11}
if $\CLin\{\Ker_\lambda T^c\vrt \lambda\in r(T)
\}=\Pi_\kappa$;
$\Pi_\Gamma$
is minimal for a simple $\pi$-symmetric
$T\in(L)$ for some $\msU$ (Corollary~\ref{cor:Mo}).
\end{rem}
\begin{rem}
$T$ with minimal $\Pi_\Gamma$ is an
operator without eigenvalues:
If $x_\mu\in\Ker_{\mu} T$,
$\mu\in\sigma_p(T)$,
$x_\lambda\in\Ker_\lambda T^c$,
$\lambda\in \msU\cup\msU^*$,
then $(\lambda-\mu)[x_\mu\,,\,x_\lambda]=0$
by $T\subset T^c$,
and then $\lambda\neq\mu$
(since $\msU\subseteq r(T_0)$,
while $\msU^*\cap\sigma_p(T)\neq\emptyset$
gives $T_0=T$) yields
$x_\mu=0$.
\end{rem}
\section{Unitary equivalence of minimal ordinary boundary
triples}
In what follows
$T$ is a closed symmetric relation in
$\msH$ with equal nonzero defect numbers,
$\Pi_\Gamma=(\msG,\Gamma)$ is an obt for $T^c$
with Weyl family $M_\Gamma$,
and
$T^\prime$ is (possibly) another closed symmetric relation
in (possibly) another Krein space $\msH^\prime$, and
$\Pi_{\Gamma^\prime}=(\msG,\Gamma^\prime)$ is
an obt for $T^{\prime c}$ with
Weyl family $M_{\Gamma^\prime}$.
(All objects
associated with $T^\prime$ are indicated by the primes
and are supposed to be defined similarly to those associated
with $T$, with the same interpretation elsewhere.)
\begin{defn}
The holomorphic elements $\{\msU,M\}$ and
$\{\msU^\prime,M^\prime\}$
are continuations of each other,
$\{\msU,M\}\leftrightarrow
\{\msU^\prime,M^\prime\}$, if
$M=M^\prime$ ($\not\equiv\mathrm{const.}$) on
each connected component
in $\msU\cap\msU^\prime$.
\end{defn}
\begin{thm}\label{thm:obt}
Let $\{\msU,\Pi_\Gamma\}$,
$\{\msU^{\prime},\Pi_{\Gamma^\prime}\}$
be minimal obt's for
$T^c$, $T^{\prime c}$ respectively.
If
$\{\msU,M_\Gamma\}\leftrightarrow
\{\msU^{\prime},M_{\Gamma^\prime}\}$
then $\Pi_{\Gamma^\prime}$ is unitarily
equivalent to $\Pi_\Gamma$.
\end{thm}
As usual
$\Pi_{\Gamma^\prime}$
is ($U$-)unitarily equivalent to $\Pi_\Gamma$
if there is
$U\in\St_1(\msH,\msH^\prime)$
such that
$\Gamma^\prime=\Gamma\tU^{-1}$,
$\tU=\bigl(\begin{smallmatrix}U & 0
\\ 0 & U \end{smallmatrix}\bigr)\co\msH_\Gamma
\lto\msH^\prime_\Gamma$.
Here
$\St_1(\msH,\msH^\prime)$ (or $\St_1(\msH)$ if
$\msH\equiv\msH^\prime$) is the set of
standard unitary operators as defined in
\cite[Definition~2.5]{Derkach09}.

To prove the theorem we use three lemmas;
obt's need not be minimal.
\begin{lem}\label{lem:iso}(\cite[Theorem~5.8]{Jursenas23})
The subspaces $T$ $(\subset\msH_\Gamma)$ and
$T^\prime$ $(\subset\msH^\prime_\Gamma)$ are isomorphic iff
$(\exists V\in\St_1(\msH_\Gamma,\msH^\prime_\Gamma))$
$\Gamma^\prime=\Gamma V^{-1}$.
\end{lem}
\begin{lem}\label{lem:W}(\cite[Theorem~6.7]{Jursenas23})
Let $T^\prime$ be $U$-unitarily equivalent to $T$
(\ie $T^\prime=\tU(T)$).
If $M_{\Gamma^\prime}=M_\Gamma$
on a nonempty subset
$\Lambda$ in $r(T)\cap r_0(T)$,
then there is $W\in\St_1(\msH_\Gamma)$ such that
\begin{SL}
\item[$\mathrm{a)}$]
$\Pi_{\Gamma^\prime}$ is $U$-unitarily equivalent to
an obt $(\msG,\Gamma W^{-1})$ for $T^c$;
\item[$\mathrm{b)}$]
$W(\lambda I\cap T^c)=\lambda I\cap T^c$,
$\lambda\in\Lambda$.
\end{SL}
\end{lem}
The third lemma is more than enough for our
primary purpose. Yet we include it
as a generalization of similar results in
\cite{Behrndt07,Hassi05,Malamud03}, in hope of
its independent interest.

First, some definitions.
Let $\msG_0$, $\msG_1$
be Hilbert spaces, $\msG=\msG_0\op\msG_1$,
$\msG^\prime=\msG_1\op\msG_0$.
Associated with (arbitrary) relations
$\GB\co\msH_\Gamma\lto\msG$,
$\GA\co\msH_\Gamma\lto\msG^\prime$
are the relations
\begin{align*}
\GB_\#=&\{(\hy,\hh)\in\msH_\Gamma\times\msG^\prime
\vrt
(\forall(\hx,\hl)\in\GB)\,
[\hx\,,\,\hy]_\Gamma=\braket{\hl\,,\,Y\hh}_\msG \}\,,
\\
\GA_\#=&\{(\hx,\hl)\in\msH_\Gamma\times\msG
\vrt
(\forall(\hy,\hh)\in\GA)\,
[\hx\,,\,\hy]_\Gamma=\braket{\hl\,,\,Y\hh}_\msG \}
\end{align*}
where
$Y=\bigl(\begin{smallmatrix}0 & -\img I_0 \\
\img I_1 & 0 \end{smallmatrix}\bigr)\co\msG^\prime\lto
\msG$, $I_0$ ($I_1$) is the identity in $\msG_0$ ($\msG_1$),
and $\braket{\,\,\cdot\,\,,\,\,\cdot\,\,}_\msG$
is the scalar product in the Hilbert sum $\msG$.

The adjoint of $\GB$ $(\GA)$ is given by
\[
(\GB)^c=(Y\GB_\#)^{-1}\quad
((\GA)^c=(Y^{-1}\GA_\#)^{-1})
\]
so that $\GB_\#$, $\GA_\#$ are closed relations.
Clearly
$\GA\subseteq\GB_\#$ iff $\GB\subseteq\GA_\#$,
and one says $\GB$, $\GA$ satisfy the
Green identity.
If $\GB$, $\GA$ are closed relations,
$\GA=\GB_\#$ iff $\GB=\GA_\#$.
\begin{defn}
Let $(A,B)$ be a dp in $\msH$.
If $\GA\subseteq\GB_\#$,
$\bar{\msD}_{\GB}=B^c$,
$\bar{\msD}_{\GA}=A^c$,
then
$((\msG,\GB),(\msG^\prime,\GA))$ is an
isometric boundary pair (ibp) for
$(A,B)$; and if also
$\GA=\GB_\#$, $\GB=\GA_\#$,
then $(\msG,\GB)\equiv((\msG,\GB),(\msG^\prime,\GA))$
is a unitary boundary pair (ubp) for $(A,B)$.
\end{defn}
Because $(A,B)$ is a dp iff such is $(B,A)$,
$((\msG,\GB),(\msG^\prime,\GA))$
is an ibp for $(A,B)$ iff such is
$((\msG^\prime,\GA),(\msG,\GB))$. The associated
Weyl families are $M_{\GB}(\lambda)=\GB(\lambda I)$
and $M_{\GA}(\lambda)=\GA(\lambda I)$,
$\gamma$-fields are
$\gamma_{\GB}(\lambda)=P(\GB_0\vrt\lambda I)^{-1}$
and
$\gamma_{\GA}(\lambda)=P(\GA_0\vrt\lambda I)^{-1}$,
all defined for $\lambda\in\bbC$.
$P\co\msH_\Gamma\lto\msH$,
$\binom{x}{y}\mapsto x$ and
$\GB_0=\{(\hx,l_0)\vrt(\exists l_1)
(\hx,\binom{l_0}{l_1})\in\GB \}$ and
$\GA_0=\{(\hy,h_1)\vrt(\exists h_0)
(\hy,\binom{h_1}{h_0})\in\GA \}$.

In agreement with a common definition
(\eg \cite[Definition~2.5]{Hassi05})
a bt for a dp $(A,B)$
is an ibp $((\msG,\GB),(\msG^\prime,\GA))$ such that
$\GB$, $\GA$ are surjective operators.
In order to see that a bt is a ubp $(\msG,\GB)$ with
surjective $\GB$, $\GA$, by
\cite[Proposition~2.1]{Popovici13}
one only needs to verify
$B\subseteq\Ker\GA$, $A\subseteq\Ker\GB$;
but the latter follows from
the Green identity and
surjectivity of $\GB$, $\GA$.

A bt for a dp $(A,B)$
is denoted by $\Pi=(\msG,\GB)$.
\begin{rem}
Similarly one defines an isometric boundary triple,
resp. a unitary boundary triple (ubt), for a dp,
if $\GB$ and $\GA$ are assumed to be operators.
In this way
a bt for a dp is a ubt such that
$\GB$, $\GA$ have closed ranges
(and are therefore surjective).
\end{rem}
\begin{lem}\label{lem:ibp}
Let $((\msG,\GB),(\msG^\prime,\GA))$
be an ibp for a dp $(A,B)$ in $\msH$ and let
\[
\phB(\lambda)=\left\{\left(y-\lambda x,
\binom{x}{y} \right)
\Bigl\vert
\binom{x}{y}\in A_0
\right\}\,,\quad A_0=\Ker\GB_0\,,
\]
\[
\GB_{10}=P_1
[(\GB)^{-1}\vrt(\{0\}\times\msG_1)]^{-1}
\]
and
$P_1\co\msG\lto\msG_1$,
$\binom{l_0}{l_1}\mapsto l_1$, and
$\lambda\in\bbC$.
Then:
\medskip{}
\begin{SL}
\item[$\mathrm{a)}$]
$\gamma_{\GB}(\lambda)-
\gamma_{\GB}(\lambda_0)=
P\phB(\lambda)(\lambda-\lambda_0)
\gamma_{\GB}(\lambda_0)$
$(\lambda_0\in\bbC)$.
\medskip{}
\item[$\mathrm{b)}$]
$M_{\GB}(\lambda)-M_{\GB}(\lambda_0)=
\GB_{10}\phB(\lambda)
(\lambda-\lambda_0)\gamma_{\GB}(\lambda_0)$
$(\lambda_0\in\bbC)$.
\medskip{}
\item[$\mathrm{c)}$]
$\GB_{10}\phB(\lambda)\subseteq
\gamma_{\GA}(\bar{\lambda})^c$
(the adjoint of $\gamma_{\GA}(\bar{\lambda})$),
with the equality iff
\medskip{}
\item[]
$\msR_{A_0-\lambda I}=
\msD_{\gamma_{\GA}(\bar{\lambda})^c}$ and
$\Ind M_{\GB}(\lambda)=\Ind M_{\GA}(\bar{\lambda})^*$.
\medskip{}
\item[$\mathrm{d)}$]
Let $A_\theta=(\GB)^{-1}(\theta)$;
$\theta$ is a lineal in $\msG$. Then
\[
(A_\theta-\lambda I)^{-1}
-(A_0-\lambda I)^{-1}\subseteq
\gamma_{\GB}(\lambda)
(\theta-M_{\GB}(\lambda))^{-1}
\GB_{10}\phB(\lambda)
\]
with the equality if
$\Ker_\lambda A_0=\Ker_\lambda \Ker\GB$.
\end{SL}
\end{lem}
\begin{proof}
For short
$\hx_\lambda=\binom{x_\lambda}{\lambda x_\lambda}=
(x_\lambda,\lambda x_\lambda)$
if $x_\lambda\in\msN_\lambda=\Ker_\lambda\msD_{\GB}$,
$\lambda\in\bbC$.

a) If $(l_0,x)\in
\gamma_{\GB}(\lambda)-\gamma_{\GB}(\lambda_0)$
then $x=x_\lambda-x_{\lambda_0}$ for some
$x_\lambda\in\msN_\lambda$, $x_{\lambda_0}\in\msN_{\lambda_0}$
such that
$(\hx_\lambda,l_0)\in\GB_0$,
$(\hx_{\lambda_0},l_0)\in\GB_0$;
hence
$(l_0,\hx)\in
\phB(\lambda)(\lambda-\lambda_0)\gamma_{\GB}(\lambda_0)$ for
$\hx=(x,y)\in A_0$,
$y=\lambda x+(\lambda-\lambda_0)x_{\lambda_0}$.

Conversely,
$(l_0,x)\in
P\phB(\lambda)(\lambda-\lambda_0)\gamma_{\GB}(\lambda_0)$
implies $(\exists y)$
$(\exists x_{\lambda_0}\in\msN_{\lambda_0})$
$(x,y)\in A_0$,
$(l_0,x_{\lambda_0})\in\gamma_{\GB}(\lambda_0)$,
$(\lambda-\lambda_0)x_{\lambda_0}=y-\lambda x$.
Because $(x+x_{\lambda_0},\lambda x+
\lambda x_{\lambda_0})=
(x,y)+\hx_{\lambda_0}\in\msD_{\GB}$, one has
$x+x_{\lambda_0}=x_\lambda\in\msN_\lambda$, so it remains
to verify $(l_0,x_\lambda)\in\gamma_{\GB}(\lambda)$:
Because $(\exists l^\prime_0\in\msG_0)$
$(l^\prime_0,x_\lambda)\in
\gamma_{\GB}(\lambda)$, it follows that
$l^\prime_0-l_0\in\GB_0(A_0)=\Ind\GB_0$, so
$(\hx_\lambda,l_0)\in\GB_0$ \ie
$(l_0,x_\lambda)\in\gamma_{\GB}(\lambda)$.

b) $M_{\GB}(\lambda)-M_{\GB}(\lambda_0)$ consists of those
$(l_0,l_1-h_1)\in\msG$ such that
$(\exists x_\lambda\in\msN_\lambda)$
$(\exists x_{\lambda_0}\in\msN_{\lambda_0})$
$(\hx_\lambda,\binom{l_0}{l_1})\in\GB$,
$(\hx_{\lambda_0},\binom{l_0}{h_1})\in\GB$;
so
$\hx=(x,y)=(x_\lambda-x_{\lambda_0},
\lambda x_\lambda-\lambda_0x_{\lambda_0})\in A_0$
and
$(\hx,l_1-h_1)\in\GB_{10}$. Since
$y-\lambda x=(\lambda-\lambda_0)x_{\lambda_0}$,
$(l_0,x_{\lambda_0})\in\gamma_{\GB}(\lambda_0)$, so
$(l_0,y-\lambda x)\in
(\lambda-\lambda_0)\gamma_{\GB}(\lambda_0)$ \ie
$(l_0,l_1-h_1)\in\GB_{10}
\phB(\lambda)(\lambda-\lambda_0)\gamma_{\GB}(\lambda_0)$.

Conversely,
$\GB_{10}\phB(\lambda)
(\lambda-\lambda_0)\gamma_{\GB}(\lambda_0)$ consists
of those $(l_0,l^\prime_1)\in\msG$ such that
$(\exists\hx=(x,y)\in A_0)$
$(l_0,\hx)\in\phB(\lambda)
(\lambda-\lambda_0)\gamma_{\GB}(\lambda_0)$
and $(\hx,l^\prime_1)\in\GB_{10}$; hence
$(\exists x_{\lambda_0}\in\msN_{\lambda_0})$
$(l_0,x_{\lambda_0})\in\gamma_{\GB}(\lambda_0)$,
$(\lambda-\lambda_0)x_{\lambda_0}=y-\lambda x$.
As in a)
this shows $x+x_{\lambda_0}=x_\lambda\in\msN_\lambda$
and $(l_0,x_\lambda)\in\gamma_{\GB}(\lambda)$.
Since $(\exists l_1,h_1\in\msG_1)$
$(\hx_\lambda,\binom{l_0}{l_1})\in\GB$,
$(\hx_{\lambda_0},\binom{l_0}{h_1})\in\GB$,
so $(l_0,l_1)\in M_{\GB}(\lambda)$,
$(l_0,h_1)\in M_{\GB}(\lambda_0)$,
$(\hx,l_1-h_1)\in\GB_{10}$, and
$l^\prime_1-(l_1-h_1)\in\Ind\GB_{10}$;
since
$\Ind\GB_{10}(=\Ind\Ind\GB)\subseteq
\GB_{10}(\lambda I)=\Ind M_{\GB}(\lambda)$,
$(l_0,l^\prime_1)\in M_{\GB}(\lambda)-M_{\GB}(\lambda_0)$.

c) Let $[\,\,\cdot\,\,,\,\,\cdot\,\,]$ be an indefinite
metric in $\msH$ and let
$\braket{\,\,\cdot\,\,,\,\,\cdot\,\,}_1$ be a
Hilbert metric in $\msG_1$.
To see $\GB_{10}\phB(\lambda)\subseteq
\gamma_{\GA}(\bar{\lambda})^c$, \ie
$(\forall(u,l_1)\in\GB_{10}\phB(\lambda))$
$(\forall(h_1,x_{\bar{\lambda}})\in
\gamma_{\GA}(\bar{\lambda}))$
$[x_{\bar{\lambda}}\,,\,u]=\braket{h_1\,,\,l_1}_1$,
use $\GA\subseteq\GB_\#$,
$u=y-\lambda x$,
$\hx=(x,y)\in A_0$,
$(\hx,\binom{0}{l_1})\in\GB$.

For the inclusion to become the equality apply
\eg \cite[Proposition~2.1]{Popovici13}.

d) $(u,y-x)\in(A_\theta-\lambda I)^{-1}-(A_0-\lambda I)^{-1}$
implies
$(\exists(l_0,l_1)\in\theta)$
$(\exists h_1\in\msG_1)$
$(\binom{x}{u+\lambda x},\binom{0}{h_1})\in\GB$,
$(\binom{y}{u+\lambda y},\binom{l_0}{l_1})\in\GB$;
so $(u,y-x)=(u,x_\lambda)\in
\msR_{A_0-\lambda I}\times\msN_\lambda$,
$(\hx_\lambda,\binom{l_0}{l_1-h_1})\in\GB$,
$(u,h_1)\in\GB_{10}\phB(\lambda)$,
$(l_0,x_\lambda)\in\gamma_{\GB}(\lambda)$,
$(l_0,l_1-h_1)\in M_{\GB}(\lambda)$ \ie
$(h_1,l_0)\in(\theta-M_{\GB}(\lambda))^{-1}$,
so
$(u,y-x)\in \gamma_{\GB}(\lambda)
(\theta-M_{\GB}(\lambda))^{-1}
\GB_{10}\phB(\lambda)$.

Conversely,
$\gamma_{\GB}(\lambda)(\theta-M_{\GB}(\lambda))^{-1}
\GB_{10}\phB(\lambda)$ consists of
$(y^\prime-\lambda x,x_\lambda)\in
\msR_{A_0-\lambda I}\times\msN_\lambda$
such that $(\exists(l_0,l_1)\in\theta)$
$(\exists h_1\in\msG_1)$
$(\binom{x}{y^\prime},\binom{0}{h_1})\in\GB$,
$(l_0,l_1-h_1)\in M_{\GB}(\lambda)$,
$(l_0,x_\lambda)\in\gamma_{\GB}(\lambda)$.
Thus $(\exists y_\lambda\in\msN_\lambda)$
$(\exists l^\prime_1\in\msG_1)$
$(\hy_\lambda,\binom{l_0}{l_1-h_1})\in\GB$,
$(\hx_\lambda,\binom{l_0}{l^\prime_1})\in\GB$,
so
$x_\lambda-y_\lambda\in\Ker_\lambda A_0$ and then
$(\hx_\lambda,\binom{l_0}{l_1-h_1})\in\GB\hsum
((\lambda I\cap A_0)\times\{0\})$. If
$\Ker_\lambda A_0=\Ker_\lambda \Ker\GB$ then
$(\hx_\lambda,\binom{l_0}{l_1-h_1})\in\GB$ and
$h_1+l^\prime_1-l_1\in\Ind\Ind\GB=\Ind\GB_{10}$.
With $y=x+x_\lambda$ and $u=y^\prime-\lambda x$,
by
$(\binom{y}{u+\lambda y},\binom{l_0}{h_1+l^\prime_1})\in\GB$
we get
$(\binom{y}{u+\lambda y},\binom{l_0}{l_1})\in\GB$, \ie
$(y^\prime-\lambda x,x_\lambda)=(u,y-x)\in
(A_\theta-\lambda I)^{-1}-(A_0-\lambda I)^{-1}$.
\end{proof}
\begin{cor}
Let $\lambda\in\bbC\setm\sigma_p(A_0)$. Then
\begin{SL}
\item[$\mathrm{a)}$]
$(A_\theta-\lambda I)^{-1}
\vrt\msR_{A_0-\lambda I}=
(A_0-\lambda I)^{-1}+
\gamma_{\GB}(\lambda)(\theta-M_{\GB}(\lambda))^{-1}
\GB_{10}\phB(\lambda)$
(\cf \cite[Theorem~5.8]{Derkach17}).
\item[$\mathrm{b)}$]
$\lambda\notin\sigma_p(A_\theta)$ iff
$\msD_{\theta \cap M_{\GB}(\lambda)}\subseteq\Ind\GB_0$.
\item[$\mathrm{c)}$]
If $\lambda\in\rho(A_0)$ then
$\lambda\in\rho(A_\theta)$ iff
$\msD_{\theta \cap M_{\GB}(\lambda)}\subseteq\Ind\GB_0$
and
$\msR_{\GB_{10}}\subseteq\msR_{\theta-M_{\GB}(\lambda)}$.
\end{SL}
\end{cor}
\begin{rem}
$\GB_{10}\subseteq\GB_1\vrt A_0$
(see Remark~\ref{rem:bt} for the definition of $\GB_1$),
with the equality iff
$\msR_{\Ind\GB}=\Ind\Ind\GB$.
\end{rem}
In the next proposition
$A_\theta$, $M_{\GB}$ formally are as above, but for
arbitrary $\GB$.
\begin{prop}\label{prop:M}
(\cite[Lemma~2.1]{Jursenas24})
\begin{SL}
\item[$\mathrm{a)}$]
$A^c_\theta\supseteq(\GB_\#)^{-1}(\theta^*)$;
if $\GB$ is a closed relation and $\theta$,
$\theta\hsum\msR_{\GB}$ are subspaces in $\msG$,
the inclusion becomes
the equality; and if moreover $\theta^*\hsum\msR_{\GB_\#}$
is a subspace in $\msG^\prime$,
$A_\theta$ is a subspace in $\msH_\Gamma$.
\item[$\mathrm{b)}$]
$M_{\GB}(\lambda)^*\supseteq\GB_\#(\bar{\lambda}I)$
$(\lambda\in\bbC)$, with the equality if
$\GB$ is a closed relation and $\msD_{\GB}-\lambda I$
has closed range; if the latter and also
$\msD_{\GB_\#}-\bar{\lambda}I$ has closed range,
then $M_{\GB}(\lambda)$ is a subspace in $\msG$.
\end{SL}
\end{prop}
\begin{rem}
By $\GB(A_\theta)=(\theta\cap\msR_{\GB})\hsum
\Ind\GB$,
in an ibp $((\msG,\GB),(\msG^\prime,\GA))$ for a
dp $(A,B)$
the set of relations
$A^\sim$: $\Ker\GB\subseteq A^\sim\subseteq B^c$
is 1-1 with the set of relations
$\theta$: $\Ind\GB\subseteq\theta\subseteq\msR_{\GB}$
via $A^\sim=A_\theta$, $\theta=\GB(A^\sim)$.
\end{rem}
\begin{proof}[Proof of Theorem~\ref{thm:obt}]
By hypothesis $\{\msU\cup\msU^\prime,M\}$,
$\{\msU^*\cup\msU^{\prime*},M_1\}$
are holomorphic elements, where
$M(\lambda)=M_\Gamma(\lambda)$ and
$M_1(\bar{\lambda})=M_\Gamma(\bar{\lambda})$
if $\lambda\in\msU$,
and
$M(\lambda)=M_{\Gamma^\prime}(\lambda)$ and
$M_1(\bar{\lambda})=M_{\Gamma^\prime}(\bar{\lambda})$
if $\lambda\in\msU^\prime$;
notice that the operator parts and multivalued parts
coincide separately.
By Lemma~\ref{lem:ibp} b), c)
$\{\msU,\lambda\mapsto
\gamma_\Gamma(\bar{\lambda}_0)^c\gamma_\Gamma(\lambda) \}$,
$\{\msU^{\prime},
\lambda\mapsto\gamma_{\Gamma^\prime}(\bar{\lambda}_0)^c
\gamma_{\Gamma^\prime}(\lambda) \}$
are holomorphic elements for each fixed
$\bar{\lambda}_0\in\msU\cap\msU^\prime$,
which are continuations
of each other via $M_1(\lambda_0)-M(\lambda)$,
and similarly for other (three) possible cases
of $\lambda,\lambda_0$ from $(\msU\cap\msU^\prime)\cup
(\msU\cap\msU^\prime)^*$.
Letting
$U\co
\msL=\Lin\{\Ker_\lambda T^c\vrt\lambda\in\msU\cup\msU^* \}
\lto
\Lin\{\Ker_\lambda T^{\prime c}\vrt
\lambda\in\msU^\prime\cup\msU^{\prime*}\}$,
$\sum_j(\gamma^s_\Gamma(\lambda_j)l_j+
\delta_\Gamma(\lambda_j)h_j)\mapsto\sum_j
(\gamma^s_{\Gamma^\prime}(\lambda_j)l_j+
\delta_{\Gamma^\prime}(\lambda_j)h_j)$---where
one uses that each $x_\lambda\in\Ker_\lambda T^c$
is of the form
$x_\lambda=\xi_\lambda+x^0_\lambda$, with
$\xi_\lambda=\gamma^s_\Gamma(\lambda)l$
($l\in\Gamma_0(\lambda I)$),
where $\gamma^s_\Gamma(\lambda)=
P(\Gamma_0\vrt\{(x,\lambda x)\vrt
x\in\Ker_\lambda T^c\om\Ker_\lambda T_0 \} )^{-1}$ is the operator
part of the $\gamma$-field, and with
$x^0_\lambda=\delta_\Gamma(\lambda)h$, where the operator
$\delta_\Gamma(\lambda)=P(\Gamma_1\vrt\lambda I\cap T_0)^{-1}$
($\Ker_\lambda T=\{0\}$)
corresponds to ($h\in$) $\Ind M_\Gamma(\lambda)=
\msG\om\Gamma_0(\bar{\lambda}I)=
\Gamma_1(\lambda I\cap T_0)$---then
writing the relation
$\gamma_\Gamma(\bar{\lambda}_0)^c\gamma_\Gamma(\lambda)$
explicitly one finds that
$U$ is a
$(J,J^\prime)$-isometric operator,
which therefore extends to
$U\in\St_1(\msH,\msH^\prime)$.

By Lemma~\ref{lem:ibp} a)
$\gamma_\Gamma(\lambda)-
\gamma_\Gamma(\mu)=
(\lambda-\mu)
(T_0-\lambda I)^{-1}
\gamma_\Gamma(\mu)$
on $\Gamma_0(\lambda I)\cap\Gamma_0(\mu I)$,
for $\lambda\in\msU$, $\mu\in\msU\cup\msU^*$.
Thus
$U(T_0-\lambda I)^{-1}
=(T^\prime_0-\lambda I^\prime)^{-1}U$ on
a dense $\Lin\{\Ker_\mu T^c\cap\msR_{T_0-\lambda I}\vrt
\mu\in\msU\cup\msU^*\}$ in $\msR_{T_0-\lambda I}$,
and
$T^\prime_0=\tU(T_0)$,
$\msD_{T^\prime_0}=U(\msD_{T_0})$.

By Lemma~\ref{lem:ibp} d)
\[
(T^\prime_1-\lambda I^\prime)^{-1}-
(T^\prime_0-\lambda I^\prime)^{-1}=
U[(T_1-\lambda I)^{-1}-
(T_0-\lambda I)^{-1}]
U^{-1}
\]
for $\lambda\in\msU$.
From here follows $T^\prime_1\supseteq\tU(T_1)$, so
$T^\prime_1=\tU(T_1)$,
$\msD_{T^\prime_1}=U(\msD_{T_1})$,
by taking the adjoint and using
$\tU\in\St_1(\msH_\Gamma,\msH^\prime_\Gamma)$.
As a result $T^\prime=\tU(T)$
on
$\msD_{T^\prime}=U(\msD_{T})$.

With Lemmas~\ref{lem:iso}, \ref{lem:W} a)
$\Gamma^\prime=\Gamma V^{-1}$,
$V=\tU W\in\St_1(\msH_\Gamma,\msH^\prime_\Gamma)$,
$W=(W_{ab})^2_{a,b=1}\in\St_1(\msH_\Gamma)$.

Consider
the operator
$K\sum_jx_j=\sum_j(W_{11}+\lambda_jW_{12})x_j$
on $\msL=\Lin\{x_j\in\Ker_{\lambda_j}T^c\}$.
By Lemma~\ref{lem:W} b)
$\Ker_{\lambda_j} T^c=
(W_{11}+\lambda_jW_{12})(\Ker_{\lambda_j}T^c)$;
hence $\msD_K=\msL=\msR_K$.
Because
\[
(W_{11}+\lambda_{j^\prime}W_{12})^c
(W_{11}+\lambda_jW_{12})x_j=
\frac{\lambda_j-\bar{\lambda}_{j^\prime}}{\lambda_j}x_j+
\frac{\bar{\lambda}_{j^\prime}}{\lambda_j}
(W_{11}+\lambda_{j^\prime}W_{12})^c
(W_{11}+\lambda_jW_{12})x_j
\]
for $\lambda_j\neq0$,
the isometric $K$
extends to $K\in\St_1(\msH)$.

On the other hand, if
$x\in \Ker(K-W_{11}-\lambda W_{12})$,
$\lambda\in r(T,T_0)$,
then
$(\exists (x_j)\subseteq\msL)$ $x-\sum^n_{j=1}x_j\lto0$
as $n\lto\infty$ strongly in $\msH$, so
$\sum_j(\lambda-\lambda_j)x_j\in\Ker W_{12}$, \ie
$\msH=\Ker W_{12}$. But then $W_{12}=0$ implies
$\Ker(K-W_{11})=\msH$,
\ie $W_{11}=K$,
and moreover $W_{22}=K$ and $W_{21}=0$.
As a result
$W=\tK$
and $\Pi_{\Gamma^\prime}$ is
$UK$-unitarily equivalent to $\Pi_\Gamma$.
\end{proof}
\section{Similarity of minimal boundary triples for
dual pairs}\label{sec:bt}
Let
$\Pi=(\msG,\GB)$ be a bt for a dp $(A,B)$ in $\msH$.
Let (see \cite{Hassi20} for block relations)
\[
r(A,B,A_0)=r(\begin{pmatrix}A & 0 \\
0 & B\end{pmatrix},
\begin{pmatrix}A_0 & 0 \\
0 & B_0\end{pmatrix} )\,,\quad B_0=A^c_0\,.
\]
\begin{defn}
A bt $\Pi\equiv\{\msU,\Pi\}$
for a dp $(A,B)$ in $\msH$ is minimal if
$\{\msU^\#,M_{\GB}\}$,
$\{\msU^\#,M_{\GA}\}$ are holomorphic
elements for both $\msU^\#=\msU$ and $\msU^\#=\msU^*$,
for an open subset
$\msU\subseteq r(A,B,A_0)$,
and
$\CLin\{\Ker_\lambda A^c
\vrt \lambda\in \msU\cup\msU^*\}=\msH=
\CLin\{\Ker_\lambda B^c
\vrt \lambda\in \msU\cup\msU^*\}$.
\end{defn}
Let $(A^\prime,B^\prime)$ be a dp
of closed relations in a Krein space $\msH^\prime$
and let
$\Pi^\prime=(\msG,\GBp)$ be
a bt for $(A^\prime,B^\prime)$ with Weyl family
$M_{\GBp}$.
\begin{defn}
$\Pi^\prime$ is ($U$-)similar to $\Pi$
if there is a homeomorphism
$U\co\msH\lto\msH^\prime$ such that
$\GBp=\GB\tU^{-1}$.
If $U\in\St_1(\msH,\msH^\prime)$, $\Pi^\prime$
is ($U$-)unitarily equivalent to
$\Pi$.
\end{defn}
\begin{thm}\label{thm:bt}
Let $\{\msU,\Pi\}$ and $\{\msU^\prime,\Pi^\prime\}$
be minimal bt's for dp's
$(A,B)$ in $\msH$ and
$(A^\prime,B^\prime)$ in $\msH^\prime$ respectively.
If
$\{\msU,M_{\GB}\}\leftrightarrow\{\msU^\prime,M_{\GBp}\}$ and
$\{\msU^*,M_{\GB}\}\leftrightarrow\{\msU^{\prime *},M_{\GBp}\}$
then
$\Pi^\prime$ is similar to $\Pi$;
if $A=B$ and $A^\prime=B^\prime$, the similarity
is unitary iff
\begin{equation}
(\GB)^{-1}\GBp=(\GA)^{-1}\GAp\,.
\label{eq:unit}
\end{equation}
\end{thm}
For $A=B$ and $A^\prime=B^\prime$,
\eqref{eq:unit} is equivalent to
$\GAp(\GBp)^{-1}=\GA(\GB)^{-1}$
by routine computation; less trivial is that
\eqref{eq:unit} holds iff
\begin{equation}
(\GB)^{-1}\GBp\co\msH^\prime_\Gamma\lto\msH_\Gamma
\;\text{is a unitary relation.}
\label{eq:unitp}
\end{equation}
(Take the adjoint and
$\GBp_\#=\GAp$, $\GB_\#=\GA$.)
Let $E=\GA(\GB)^{-1}(=[\GB(\GB)^cY]^{-1})$,
a homeomorphism
$\msG\lto\msG^\prime$.
From
$\GA=E\GB$ follows $M_{\GA}(\lambda)=E(M_{\GB}(\lambda))$,
$\lambda\in\bbC$. Because
$M_{\GA}(\lambda)=M_{\GB}(\bar{\lambda})^*$
for $\lambda\in r_0(A)(=r^*_0(A))$,
by \eqref{eq:unit}
$\{\msU,M_{\GB}\}\leftrightarrow\{\msU^\prime,M_{\GBp}\}$
iff
$\{\msU^*,M_{\GB}\}\leftrightarrow\{\msU^{\prime *},M_{\GBp}\}$.
Hence the corollary.
\begin{cor}\label{cor:bt}
Let $\{\msU,\Pi\}$ and $\{\msU^\prime,\Pi^\prime\}$
be minimal bt's for dp's
$(A,A)$ in $\msH$ and
$(A^\prime,A^\prime)$ in $\msH^\prime$ respectively,
such that \eqref{eq:unit} holds.
If
$\{\msU,M_{\GB}\}\leftrightarrow\{\msU^\prime,M_{\GBp}\}$
then
$\Pi^\prime$ is unitarily equivalent to $\Pi$.
\end{cor}
\begin{rem}
In the corollary $\{\msU^\#,M_{\GB}\}$
is a holomorphic element iff (via $E$) such is
$\{\msU^\#,M_{\GA}\}$; \ie if
$\{\msU,M_{\GB}\}$, $\{\msU^*,M_{\GA}\}$
are holomorphic elements
(\eg if $\msU\subseteq\rho(A_0)$,
see Lemma~\ref{lem:ibp}), then so are
$\{\msU^*,M_{\GB}\}$, $\{\msU,M_{\GA}\}$,
and conversely.
\end{rem}
\begin{proof}[Proof of Theorem~\ref{thm:bt}]
Given a dp $(A,B)$ in a $J$-space $\msH$,
the block diagonal relation
$T=\bigl(\begin{smallmatrix}A & 0 \\
0 & B \end{smallmatrix}\bigr)$ is
closed symmetric both
in the Krein space
$(\msH_\Gamma,[\,\,\cdot\,\,,\,\,\cdot\,\,]_\Gamma)$
and in the Krein space
$(\msH_{\hGm}=\msH\op\msH,[\,\,\cdot\,\,,\,\,\cdot\,\,]_{\hGm})$
with canonical symmetry
$J_{\hGm}\binom{x}{y}=\binom{Jy}{Jx}$
and indefinite metric
$\bigl[\binom{x_1}{y_1}\,,\,\binom{x_2}{y_2}\bigr]_{\hGm}=
[x_1\,,\,y_2]+[y_1\,,\,x_2]$.

Consider
{\small
\[
\Gamma=
\left\{
\left(
\left(\begin{pmatrix}
x \\ y\end{pmatrix},
\begin{pmatrix}
x^\prime \\ y^\prime\end{pmatrix} \right),
\left(\begin{pmatrix}
l_0 \\ h_1\end{pmatrix},
\begin{pmatrix}
h_0 \\ l_1\end{pmatrix} \right)
\right)\Bigl\vert
\left(\begin{pmatrix}x \\ x^\prime\end{pmatrix},
\begin{pmatrix}l_0 \\ l_1\end{pmatrix} \right)
\in\GB\,;\,
\left(\begin{pmatrix}y \\ y^\prime\end{pmatrix},
\begin{pmatrix}h_1 \\ h_0\end{pmatrix} \right)
\in\GA \right\}
\]
}
as a relation from the Krein space
$\msK=
(\msH_{\hGm}\op\msH_{\hGm},[\,\,\cdot\,\,,\,\,\cdot\,\,]_\msK)$
with canonical symmetry
$\bigl(\begin{smallmatrix}0 & -\img J_{\hGm}
\\ \img J_{\hGm} & 0 \end{smallmatrix}\bigr)$ and
indefinite metric
\[
\left[\left(\begin{pmatrix}x \\ y\end{pmatrix},
\begin{pmatrix}x^\prime \\ y^\prime\end{pmatrix}
\right),
\left(\begin{pmatrix}x_* \\ y_*\end{pmatrix},
\begin{pmatrix}x^\prime_* \\ y^\prime_*\end{pmatrix}
\right)\right]_\msK
=
\img\left[\begin{pmatrix}x \\ y\end{pmatrix},
\begin{pmatrix}x^\prime_* \\ y^\prime_*\end{pmatrix}
\right]_{\hGm}-
\img\left[\begin{pmatrix}x^\prime \\ y^\prime\end{pmatrix},
\begin{pmatrix}x_* \\ y_*\end{pmatrix}\right]_{\hGm}
\]
to the Krein space $\msG_{\Gamma}$.
(Just like $\msG_\Gamma$,
$\msK$ is defined according to the same rule
\cite[Section~2.1, Eq.~(1.4)]{Azizov89}.)

It is a simple computational task to verify
that in general
(\ie for arbitrary $\GB$, $\GA$)
$(\Gamma^c)^{-1}$ is given the same as $\Gamma$, but with
$\GA_\#$, $\GB_\#$ in place of $\GB$, $\GA$, respectively;
\[
\begin{split}
&\msD_{\Gamma}=\begin{pmatrix}\msD_{\GB} & 0 \\
0 & \msD_{\GA}\end{pmatrix}\,,\quad
\Ker\Gamma=\begin{pmatrix}\Ker\GB & 0 \\
0 & \Ker\GA\end{pmatrix}\,,
\\
&\msR_{\Gamma}=\begin{pmatrix}0 & \msR_{\GA} \\
\msR_{\GB} & 0 \end{pmatrix}\,,\quad
\Ind\Gamma=\begin{pmatrix}0 & \Ind\GA \\
\Ind\GB & 0 \end{pmatrix}\,.
\end{split}
\]

In case of bt's therefore
$\Gamma\co\msK\lto\msG_\Gamma$ is unitary
surjective and
$\Pi_\Gamma=(\msG,\Gamma)$
is an obt for $T^c=\bigl(\begin{smallmatrix}B^c & 0 \\
0 & A^c \end{smallmatrix}\bigr)$.
The Weyl family
$M_\Gamma$ associated with $\Pi_\Gamma$ reads
\[
M_\Gamma(\lambda)=
\begin{pmatrix}0 & M_{\GA}(\lambda) \\
M_{\GB}(\lambda) & 0
\end{pmatrix}\,,\quad\lambda\in\bbC
\]
and moreover
$M_{\GA}(\lambda)=M_{\GB}(\bar{\lambda})^*$
if $\lambda\in r_0(A)\cap r_0(B)$.

In parallel one constructs an obt
$\Pi_{\Gamma^\prime}=(\msG,\Gamma^\prime)$
for $T^{\prime c}=\bigl(\begin{smallmatrix}
B^{\prime c} & 0 \\
0 & A^{\prime c} \end{smallmatrix}\bigr)$
with Weyl family $M_{\Gamma^\prime}$.

Because $\{\msU,\Pi_\Gamma\}$,
$\{\msU^\prime,\Pi^\prime_\Gamma\}$
are minimal obt's for
$T^c$, $T^{\prime c}$ respectively,
where open subsets
$\msU\subseteq r(T,T_0)$,
$\msU^\prime\subseteq r(T^\prime,T^\prime_0)$,
and self-adjoint relations
$T_0=\bigl(\begin{smallmatrix}A_0 & 0 \\
0 & B_0 \end{smallmatrix}\bigr)=\Ker\Gamma_0$,
$T^\prime_0=\bigl(\begin{smallmatrix}A^\prime_0 & 0 \\
0 & B^\prime_0 \end{smallmatrix}\bigr)=\Ker\Gamma^\prime_0$,
and because
$\{\msU,M_\Gamma\}\leftrightarrow
\{\msU^\prime,M_{\Gamma^\prime}\}$,
from Theorem~\ref{thm:obt} follows that
$\Pi_{\Gamma^\prime}$ is $U$-unitarily equivalent
to $\Pi_\Gamma$.

Write $U\in\St_1(\msH_{\hGm},\msH^\prime_{\hGm})$
in matrix form $(U_{ij})^2_{i,j=1}$:
\[
\begin{split}
U^c_{11}U_{22}+U^c_{21}U_{12}=I\,,\quad
U_{11}U^c_{22}+U_{12}U^c_{21}=I^\prime\,,
\\
U^c_{11}U_{21}+U^c_{21}U_{11}=0\,,\quad
U_{11}U^c_{12}+U_{12}U^c_{11}=0\,,
\\
U^c_{12}U_{22}+U^c_{22}U_{12}=0\,,\quad
U_{21}U^c_{22}+U_{22}U^c_{21}=0\,.
\end{split}
\]
Then $\Gamma^\prime=\Gamma\tU^{-1}$ yields:
$(\hx,\hl)\in\GB$,
$(\hy,\hh)\in\GA$ iff
$(\tU_{11}\hx+\tU_{12}\hy,\hl)\in\GBp$,
$(\tU_{21}\hx+\tU_{22}\hy,\hh)\in\GAp$.
Thus
$\GB\subseteq\GBp\tU_{11}$,
$\tU_{21}(B^c)\subseteq B^\prime$,
and
$\GA\subseteq\GAp\tU_{22}$,
$\tU_{12}(A^c)\subseteq A^\prime$.

From $\tU_{21}(B^c)\subseteq B^\prime$ follows
$U_{21}(\Ker_\lambda B^c)
\subseteq\Ker_\lambda B^\prime$
($\lambda\in\bbC$); \ie
$U_{21}=0$ by minimality of $\Pi$, $\Pi^\prime$.
Similarly
$U_{12}=0$, so
$U_{11}$ and $U_{22}=(U^c_{11})^{-1}$ are homeomorphisms
$\msH\lto\msH^\prime$.
Then
$\GB_\#\subseteq\GBp_\#\tU_{22}$ implies
$\GB\supseteq\GBp\tU_{11}$, and then
$\GB=\GBp\tU_{11}$
(equivalently $\GA=\GAp\tU_{22}$), and
this proves that $\Pi^\prime$ is $U_{11}$-similar to $\Pi$.

Assume $A=B$, $A^\prime=B^\prime$;
$U_{11}$ is a homeomorphism as above.
From
$\GBp=\GB\tU^{-1}_{11}$ and
$\GAp=\GA\tU^{-1}_{22}$ (with
$U_{22}=(U^c_{11})^{-1}$) follows
\begin{align*}
(\GB)^{-1}\GBp=&(\tU_{11}\vrt A^c)^{-1}\hsum
(\{0\}\times A)\,,
\\
(\GA)^{-1}\GAp=&(\tU_{22}\vrt A^c)^{-1}\hsum
(\{0\}\times A)
\end{align*}
which proves necessity of \eqref{eq:unit}
in order that $U_{11}=U_{22}$ should hold.
Conversely, if \eqref{eq:unit} then by the above
$(\tU_{22}-\tU_{11})(A^c)\subseteq\tU_{22}(A)$;
hence $(U_{22}-U_{11})(\Ker_\lambda A^c)\subseteq
U_{22}(\Ker_\lambda A)$ ($\lambda\in\bbC$),
and by minimality $U_{22}=U_{11}$.
\end{proof}
\begin{rem}\label{rem:bt}
By correspondence
$\Pi\leftrightarrow\Pi_\Gamma$
established in the proof,
every dp has a bt:
$\dim\Ker_\img J_{\hGm}T^c=\dim\Ker_{-\img} J_{\hGm}T^c$,
because $\Ker_{\pm\img}J_{\hGm}T^c=\pm\img
((JA^c)^{-1}\cap-JB^c )$.
By a simple
computation, moreover, all boundary
value spaces defined in
\cite{Derkach17} have their analogues for a dp;
particularly the pair $(\msG,\Gamma)$ is an ibp
for $T^c$ iff $((\msG,\GB),(\msG^\prime,\GA))$
is an ibp for a dp $(A,B)$.
With
$\GB_1=P_1\GB$
($P_1$ as in Lemma~\ref{lem:ibp}),
$\GA_1=P^\prime_1\GA$
($P^\prime_1\co\msG^\prime\lto\msG_0$,
$\binom{h_1}{h_0}\mapsto h_0$),
$\msD_{\Gamma_i}=\msD_\Gamma$
($i=0,1$), and
\[
T_i=\Ker\Gamma_i=\begin{pmatrix}A_i & 0 \\
0 & B_i\end{pmatrix}\,,\quad
A_i=\Ker\GB_i\,,\quad
B_i=\Ker\GA_i\,,
\]
\[
\msR_{\Gamma_0}=\msR_{\GB_0}\times\msR_{\GA_0}\,,\quad
\msR_{\Gamma_1}=\msR_{\GA_1}\times\msR_{\GB_1}\,,
\]
\[
\Ind\Gamma_0=\Ind\GB_0\times\Ind\GA_0\,,\quad
\Ind\Gamma_1=\Ind\GA_1\times\Ind\GB_1
\]
the following equivalences hold:
\\
$-$ $(\msG,\Gamma)$ is an $AB$-generalized boundary pair (bp)
for $T^c$
(\ie ibp with
$T_0$ self-adjoint and $\bar{\msR}_{\Gamma_0}=\msG$)
iff
$((\msG,\GB),(\msG^\prime,\GA))$ is an
$AB$-generalized bp for $(A,B)$
(\ie ibp with
$A_0=B^c_0$, $B_0=A^c_0$,
$\bar{\msR}_{\GB_0}=\msG_0$,
$\bar{\msR}_{\GA_0}=\msG_1$);
\\
$-$ $(\msG,\Gamma)$ is a $q$-bt for $T^c$
(\ie $AB$-generalized bp with
$\bar{\msR}_{\Gamma_1}=\msG$) iff
$((\msG,\GB),(\msG^\prime,\GA))$ is a
$q$-bt for $(A,B)$
(\ie $AB$-generalized bp with
$\bar{\msR}_{\GB_1}=\msG_1$,
$\bar{\msR}_{\GA_0}=\msG_0$), see \cite{Behrndt23};
\\
$-$ $(\msG,\Gamma)$ is an $ES$-generalized bp for $T^c$
(\ie ubp with $T_0$ essentially
self-adjoint)
iff
$((\msG,\GB),(\msG^\prime,\GA))$ is an
$ES$-generalized bp for $(A,B)$
(\ie ubp with
$\bar{A}_0=B^c_0$);
\\
$-$ $(\msG,\Gamma)$ is an $S$-generalized bp for $T^c$
(\ie $ES$-generalized bp with $T_0$ self-adjoint) iff
$((\msG,\GB),(\msG^\prime,\GA))$ is an
$S$-generalized bp for $(A,B)$
(\ie $ES$-generalized bp with
$A_0=B^c_0$ and $B_0=A^c_0$);
\\
$-$ $(\msG,\Gamma)$ is a $B$-generalized bp
for $T^c$
(\ie $S$-generalized bp with $\msR_{\Gamma_0}=\msG$)
iff
$((\msG,\GB),(\msG^\prime,\GA))$ is a
$B$-generalized bp for $(A,B)$
(\ie $S$-generalized bp with
$\msR_{\GB_0}=\msG_0$, $\msR_{\GA_0}=\msG_1$).
\end{rem}
In conclusion, the method used in the
proof of Theorem~\ref{thm:bt} has an application in
various boundary value problems, if one is due to
transfer results from obt's to bt's for dp's.
In particular generalized resolvents
generated by minimal extensions
are worth to mention, see
\cite[Proposition~3.5]{Hassi05}.
In what follows, however, we concentrate on examples
of \eqref{eq:unit}.
\begin{exam}
By \cite[Theorem~6.2]{Arlinskii05}
a bounded everywhere defined contraction $T$ in a Hilbert
space $\msH$ is a qsc-operator if
$\msD_{T\cap T^*}$ is nontrivial; it is
a qsc-extension of a closed symmetric contraction
$A=T\vrt\msH\om\msN$ in $\msH$, if $\msN=\bar{\msN}\supseteq
\msR_{T^*-T}$.
Associated with
a qsc-operator $T$ is the
$Q$-function
$Q_T(\lambda)=P_\msN(T-\lambda I)^{-1}\vrt\msN$,
$\lambda\in\msU=\{\lambda\vrt\abs{\lambda}>1\}$
($P_\msN$ is an orthogonal projection in $\msH$ onto $\msN$),
which is the Weyl function $M_{\GB}(\lambda)$ of
a bt $\Pi$ for a dp
$(A,A)$, with $\msG_0=\msG_1=\msN$,
$\GB\binom{x}{y}=\binom{P_\msN(Tx-y)}{P_\msN x}$,
$\binom{x}{y}\in A^*=T\hsum(\{0\}\times\msN)$, so that
$\GA=\GB+\binom{(T^*-T)P}{0}$;
$T$ is $\msN$-minimal if
$\CLin\{\Ker_\lambda A^*=(T-\lambda I)^{-1}(\msN)\vrt\,
\lambda\in\msU \}=\msH$. To see that
$Q_T$ on $\msU$ ($\subseteq r(A,A,A_0=T)$)
determines an $\msN$-minimal qsc-operator $T$
uniquely up to unitary equivalence, by
\eqref{eq:unit} one only needs
to observe that
$(T^*-T)x=(T^{\prime*}-T^\prime)x^\prime$
if
$\bigl(\binom{x^\prime}{y^\prime},
\binom{x}{y} \bigr)\in(\GB)^{-1}\GBp$, where
$T^\prime$ is an $\msN$-minimal qsc-operator
in a Hilbert space $\msH^\prime\supseteq\msN$,
with $\{\msU,Q_T\}\leftrightarrow
\{\msU,Q_{T^\prime}\}$ and a bt $\Pi^\prime$
defined similarly.
\end{exam}
\section{Example: $D$-boundary triple
in a $\Pi_\kappa$-space}\label{sec:Dbt}
\begin{defn}
Let $A$ be a closed symmetric relation in
a Krein space $\msH$ and let the Hilbert spaces
$\msG_1\subseteq\msG_0$ and $\msG_2=\msG_0\om\msG_1$.
A $D$-bt for $A^c$ is a bt
$\Pi=(\msG,\GB)$ for $(A,A)$ such that
$\GA=E\GB$, where
$E(l_0,l_1)=(E_1l_0,\img E_2l_0+l_1)$,
$\msD_E=\msG$, where $E_1$ ($E_2$) is an
orthogonal projection in $\msG_0$ onto
$\msG_1$ ($\msG_2$).
A $D$-bt is minimal if such is a bt $\Pi$.
\end{defn}
In a $D$-bt
$A_0$ ($B_1$) is a maximal $J$-symmetric relation,
\cf \cite[Proposition~3.11]{Mogilevski06},
with zero negative defect number;
to avoid trivialities $A$ is therefore
non-maximal $J$-symmetric whenever a $D$-bt
is considered.

Corollary~\ref{cor:bt} for minimal $D$-bt's generalizes
Theorem~\ref{thm:obt}, since $E$ is injective.
We apply this corollary to symmetric operators in
a $\Pi_\kappa$-space, a separable Pontryagin space with
$\kappa$ negative squares (and indefinite metric
$[\,\,\cdot\,\,,\,\,\cdot\,\,]$ and induced
norm $\norm{\,\,\cdot\,\,}$).

Below, $A$ is always a closed $\pi$-symmetric operator.
Following \cite{Jursenas24,Jursenas23}
$A\in(P)$ if the lineal
$\msD_A+\msR_A$ is dense in $\Pi_\kappa$;
$A\in(P)$ iff
$\Ker_{\lambda}A^c\cap\Ker_{\mu}A^c=\{0\}$
for some and then all
$\lambda,\mu\in\sigma_p(A^c)$,
$\lambda\neq\mu$.
\begin{defn}(\cite{Jursenas24})
$A\in(LP)$
if $A\in(P)$ has the (Langer)
property $(L)$:
There is a canonical decomposition
$\Pi_\kappa=\Pi^-[\op]\Pi^+$ in which the negative
subspace $\Pi^-\subseteq\msD_A$.
\end{defn}
A maximal $\pi$-dissipative operator is
of class $(LP)$. From \cite[Theorem~1.3]{Jursenas24}
(where $\Gamma_A$ is defined precisely) follows
\begin{prop}\label{prop:dsptT}
For a $\pi$-dissipative extension $A^\sim$
of $A\in(L)$
\[
\bbC^-\cap\Gamma_A\subseteq r(A^\sim)\quad
(\text{equivalently  }
\bbC^-\cap\sigma_p(A^\sim)\subseteq C_A)
\]
for some open subset
$\Gamma_A=\bbC_*\setm C_A\supseteq
\{\lambda\vrt\abs{\im\lambda}>t_0
\norm{AP^-} \}$,
$t_0\approx1.84$, where
$P^-$ is the canonical projection onto
$\Pi^-\subseteq\msD_A$;
the subset $\Gamma_A$ is symmetric:
$\Gamma_A=-\Gamma_A=\Gamma_{-A}=\Gamma^*_A$.
Moreover $A^\sim$ is an operator if $A\in(P)$.
\end{prop}
\begin{rem}
$\bbC^-\cap\{\lambda\vrt\abs{\im\lambda}>
2\norm{A^\sim P^-} \}\subseteq\rho(A^\sim)$
for a maximal $\pi$-dissipative
operator $A^\sim$ (\cite[Theorem~2.2.9]{Azizov89}).
\end{rem}
In a $\Pi_0$-space the Weyl function of a $D$-bt
is defined
in \cite{Mogilevski11,Mogilevski09,Mogilevski06}
on $\bbC^+=\rho(A_0)$,
in order to have it holomorphic.
For $A\in(LP)$
in a $\Pi_\kappa$-space $\rho(A_0)$ is characterized
thus.
\begin{prop}\label{prop:dsptT2}
Let $A\in(LP)$, $\Pi^-\subseteq\msD_A$,
and let $A^\sim\supseteq A$
be a closed maximal $\pi$-symmetric operator
with defect numbers $(d_+,0)$, $d_+\geq1$.
Then
\[
\bbR\cap\rho(A^\sim)=\bbR\cap r(A^\sim)\,,\quad
\bbC^+\cap\rho(A^\sim)\supseteq
\bbC^+\cap\Gamma_A\,,\quad
\bbC^-\cap\rho(A^\sim)\subseteq\bbC^-\cap C_A
\]
and $\rho(A^\sim)\cap\rho^*(A^\sim)=\emptyset$.
\end{prop}
\begin{proof}
The first two claims follow from
$\bbR\cap\sigma_r(A^\sim)=\emptyset$
(\cite[Corollary~2.2.17]{Azizov89}) and
Proposition~\ref{prop:dsptT}.
For the last two claims
it suffices to verify
that $\msR_{A^\sim-\lambda I}\neq\Pi_\kappa$
for all $\lambda\in\bbC^-\cap\Gamma_A$, as
the latter would imply
$\bbC^-\cap\sigma_r(A^\sim)\supseteq\bbC^-\cap\Gamma_A$.

If $\msR_{A^\sim-\lambda I}=\Pi_\kappa$
for some and then for all $\lambda\in\bbC^-\cap\Gamma_A$,
then $\bbC^-\cap\Gamma_A\subseteq\bbC^-\cap\rho(A^\sim)$
so $\Gamma_A\subseteq\rho(A^\sim)\cap\rho^*(A^\sim)$;
as in \cite[Theorem~1.5.5(iv)]{Behrndt20}
the latter implies $A^\sim=(A^\sim)^c$, a contradiction
to $d_+\neq0$.
\end{proof}
\begin{rem}
If $\kappa=0$ then
$C_A=\emptyset$, so
$\bbR\cap r(A^\sim)=\emptyset$
(\eg \cite[Proposition~3.3]{Schmudgen12}),
$\rho(A^\sim)=\bbC^+$.
\end{rem}
\begin{defn}
$A\in(LP)$ is simple if
$\CLin\{\Ker_\lambda A^c\vrt
\lambda\in\Gamma_A \}=\Pi_\kappa$;
$\Pi^-\subseteq\msD_A$.
\end{defn}
\begin{prop}\label{prop:CLLP}(\cf \cite{Azizov03})
$A\in(LP)$ is simple iff
$\CLin\{\Ker_\lambda A^c\vrt
\lambda\in\bbC_* \}=\Pi_\kappa$.
\end{prop}
\begin{proof}
As in \cite[Lemma~1.6.11]{Behrndt20}
consider
$V=I+(\lambda-\mu)(A^\sim-\lambda I)^{-1}$
for any maximal $\pi$-dissipative
(maximal $\pi$-accumulative) $A^\sim\supseteq A$ and
$\lambda,\mu\in\rho(A^\sim)$; in the canonical
decomposition
$\Pi_\kappa=\Pi^-[\op]\Pi^+$,
$\Pi^-\subseteq\msD_A$, by Proposition~\ref{prop:dsptT}
$\bbC^-\cap\Gamma_A\subseteq\rho(A^\sim)$
($\bbC^+\cap\Gamma_A\subseteq\rho(A^\sim)$).
$V$ maps $\Ker_\mu A^c$ bijectively
onto $\Ker_\lambda A^c$, that is
$x_\lambda=Vx_\mu\in\Ker_\lambda A^c$ if
$x_\mu\in\Ker_\mu A^c$,
and conversely.

If $x\in\bigcap_{\lambda\in\Gamma_A}
\msR_{A-\lambda I}$ then the function
$\lambda\mapsto[x_\lambda\,,\,x]$
holomorphic on $\bbC^-\cap\rho(A^\sim)$
($\bbC^+\cap\rho(A^\sim)$) vanishes on
$\bbC^-\cap\Gamma_A$ ($\bbC^+\cap\Gamma_A$),
so it
must vanish on all of $\bbC^-\cap\rho(A^\sim)$
($\bbC^+\cap\rho(A^\sim)$);
that is
$\CLin\{\Ker_\lambda A^c\vrt
\lambda\in\Gamma_A \}=
\CLin\{\Ker_\lambda A^c\vrt
\lambda\in\bbC_*\cap\rho(A^\sim) \}=\Pi_\kappa$.

As the final ingredient
$\rho(A^\sim)$ contains $\bbC^-$ ($\bbC^+$)
with a possible exception of at most $\kappa$
normal eigenvalues of $A^\sim$;
$A^\sim$ is an operator by Proposition~\ref{prop:dsptT}.
\end{proof}
\begin{cor}\label{cor:Mo}
For $R$ sufficiently large,
the Weyl function on $\{\lambda\vrt\im\lambda>R \}$
of a closed simple
$\pi$-symmetric operator of class $(LP)$
determines a $D$-bt
uniquely up to unitary equivalence.
\end{cor}
\begin{proof}
A $D$-bt $\{\bbC^-\cap\Gamma_A,\Pi\}$
$(\Pi^-\subseteq\msD_A)$ of a simple
(non-maximal) $\pi$-symmetric operator $A\in(LP)$
is minimal:
Since $A_0$ ($B_0$) is a maximal $\pi$-symmetric
(maximal $\pi$-dissipative) operator,
by Propositions~\ref{prop:dsptT}, \ref{prop:dsptT2}
$r(A,A,A_0)\cap\Gamma_A=\bbC^-\cap\Gamma_A$.
Now apply Proposition~\ref{prop:CLLP}.
\end{proof}
\begin{rem}
$\{\bbC^-\cap\Gamma_A,M_{\GB} \}$
$(\Pi^-\subseteq\msD_A)$ is
the holomorphic element; the representing
pair $\{\Phi(\lambda),\Psi(\lambda)\}$
for the closed operator
$M_{\GB}(\lambda)=E^{-1}(M_{\GA}(\lambda))$
is
$\Phi(\lambda)=I_1-\img E_2M_{\GA}(\lambda)$,
$\Psi(\lambda)=E_1M_{\GA}(\lambda)$ (equivalently
$M_{\GB}(\lambda)=E_1(M_{\GA}(\lambda)^{-1}-
\img E_2 )^{-1}$). One could equally
formulate the corollary for Weyl function
$M_{\GB}$
on $\{\lambda\vrt\im\lambda<-R\}$
($R\geq t_0\norm{AP^-}$).
\end{rem}
\begin{rem}
From the proofs of Propositions~\ref{prop:dsptT2}
and \ref{prop:CLLP} one finds that
in a $\Pi_0$-space the property $(P)$
in Corollary~\ref{cor:Mo}
can be dropped, since $\rho(B_0)=\bbC^-$;
because also $\Gamma_A=\bbC_*$,
the result is therefore
\cite[Theorem~3.14]{Mogilevskii20}.
\end{rem}
\section{Example: Boundary triple determined by the
fractional linear transformation of Nevanlinna
function}\label{sec:lft}
Consider a closed symmetric
operator $A$ in a Krein space $\msH$ with equal
nonzero defect numbers.
Let $(\msL,\mGm)$ be an obt for $A^c$ with Weyl family
$M_{\mGm}$.
Suppose the Hilbert spaces $\msL$ and $\msG_0$
are isomorphic
and let $W$ be a homeomorphism
$\msL_\Gamma\lto\msG$; $\msG_1=\msG_0$.
Then $\Pi=(\msG,\GB)$, where $\GB=W\mGm$ and
$\GA=(W^c)^{-1}\mGm$, is a bt for a dp $(A,A)$
with Weyl family
$M_{\GB}(\lambda)=W(M_{\mGm}(\lambda))$;
$W^c$ is the adjoint of $W$, the latter being
considered as the operator from the Krein space
$\msL_\Gamma$ to the $\bigl(\begin{smallmatrix}0 & -\img I_0 \\
\img I_0 & 0 \end{smallmatrix}\bigr)$-space
$\msG$ (\ie the Krein space
$\msG=(\msG_0)_\Gamma$). It is supposed that
$\Pi$ is minimal.

Let $A^\prime$ be a closed symmetric
operator in a Krein space $\msH^\prime$
with the same defect numbers as those of $A$;
$(\msL^\prime,\mGm^\prime)$ is an obt for $A^{\prime c}$
with Weyl faimly $M_{\mGm^\prime}$,
$W^\prime$ is a homeomorphism
$\msL^\prime_\Gamma\lto\msG$.
Let
$\Pi^\prime=(\msG,\GBp)$,
$\GBp=W^\prime\mGm^\prime$,
$\GAp=(W^{\prime c})^{-1}\mGm^\prime$,
be a minimal bt for a dp $(A^\prime,A^\prime)$ with
Weyl family
$M_{\GBp}(\lambda)=
W^\prime(M_{\mGm^\prime}(\lambda))$.

Under these conditions
\eqref{eq:unitp} holds
iff
a homeomorphism $V=W^{\prime\,-1}W\co
\msL_\Gamma\lto\msL^\prime_\Gamma$ belongs
to $\St_1(\msL_\Gamma,\msL^\prime_\Gamma)$.

We illustrate the necessary and sufficient
condition with
\[
W=\begin{pmatrix}
K^{-1}B & -K^{-1} \\
K^*+CK^{-1}B & -CK^{-1}
\end{pmatrix}
\]
(see \cite[Remark~3.16]{Hassi13}), where
$K\in\msB(\msG_0,\msL)$ is bijective,
$B\in\msB(\msL)$, $C=\msB(\msG_0)$.
The operator $W^\prime$ is defined analogously,
with the associated objects indicated by the
primes. Then
\begin{prop}\label{prop:V}
$V\in\St_1(\msL_\Gamma,\msL^\prime_\Gamma)$ iff
\begin{equation}
\begin{cases}
\im(C^\prime-C)=0\,,\quad
\msR_{\im B}\subseteq K(\msD_{C\cap C^\prime})\,,
\\
\im B^\prime=(K^\prime K^{-1})(\im B)
(K^\prime K^{-1})^*\,.
\end{cases}
\label{eq:sysV}
\end{equation}
\end{prop}
(Note that $C\cap C^\prime=\re C\cap\re C^\prime$
if $\im C=\im C^\prime$.)
\begin{proof}
By definition
\[
V=\begin{pmatrix}
V_1 & V_2 \\
B^\prime V_1-K^\prime K^{-1}B &
K^\prime K^{-1}+B^\prime V_2
\end{pmatrix}\,,
\]
\[
V_1=(KK^{\prime\,-1})^*-V_2B\,,
\quad
V_2=K^{\prime\,*\,-1}(C^\prime-C)K^{-1}
\]
and by routine computation
$V\in\St_1(\msL_\Gamma,\msL^\prime_\Gamma)$ iff
\[
\begin{cases}
\msR_{\im B}\subseteq\Ker V_2\quad
\text{(or equivalently $\msR_{V^*_2}\subseteq
\msD_{B\cap B^*}$)}\,,
\\
\msR_{\im B^\prime}\subseteq\Ker V^*_2\quad
\text{(or equivalently $\msR_{V_2}\subseteq
\msD_{B^\prime\cap B^{\prime *}}$)}\,,
\\
\im(V_2KK^{\prime\,-1})=0\,,\quad
\im B^\prime=(K^\prime K^{-1})(\im B)
(K^\prime K^{-1})^*\,;
\end{cases}
\]
substituting $V_2$ in the latter system gives
\eqref{eq:sysV}.
\end{proof}
If $\Pi^\prime$ is unitarily equivalent to $\Pi$,
$\mGm^{\prime\,-1}(\theta^\prime)$
$(\theta^\prime\subseteq\msL^\prime_\Gamma)$
is unitarily equivalent to
$\mGm^{-1}(\theta)$
$(\theta=V^{-1}(\theta^\prime))$;
for $V$ as in Proposition~\ref{prop:V} this means
\[
\theta^\prime=B^\prime+K^\prime
[\re(C^\prime-C)
+K^*(\theta-B)^{-1}K]^{-1}
K^{\prime *}\,.
\]
In particular
\begin{SL}
\item[a)]
$A_0=\Ker(\mGm_1-B\mGm_0)$
(resp. $B_0=\Ker(\mGm_1-B^*\mGm_0)$) and
$A^\prime_0=\Ker(\mGm^\prime_1-B^\prime\mGm^\prime_0)$
(resp. $B^\prime_0=
\Ker(\mGm^\prime_1-B^{\prime *}\mGm^\prime_0)$)
are unitarily equivalent, but
\item[b)]
$\mT_0=\Ker\mGm_0$ (\ie $\theta=\{0\}\times\msL$)
gives
$\theta^\prime=B^\prime+K^\prime[\re(C^\prime-C)]^{-1}
K^{\prime *}=\theta^{\prime*}$.
\end{SL}

Further simplification is possible in the
Hilbert space setting.
\begin{prop}
Suppose $\Pi^\prime$ is similar to $\Pi$.
If $A$ $(A^\prime)$ is considered as a closed
symmetric operator in a Hilbert space $\msH$ $(\msH^\prime)$,
then the first two conditions in the system \eqref{eq:sysV}
are sufficient in order that the third one should hold.
\end{prop}
\begin{proof}
By hypothesis
$\GBp=\GB\tU^{-1}_{11}$ for a homeomorphism
$U_{11}\co\msH\lto\msH^\prime$, which we are willing to
show can be chosen unitary, \ie
$\im B^\prime=(K^\prime K^{-1})(\im B)
(K^\prime K^{-1})^*$.

Since $\GB=W^\prime V\mGm$, $\GBp=W^\prime\mGm^\prime$,
so $\mGm^\prime=V\mGm\tU^{-1}_{11}$.
Since $\mGm^\prime\co
\msH^\prime_\Gamma\lto\msL^\prime_\Gamma$ is unitary,
taking the adjoint and the inverse
yields $\mGm^\prime=(V^c)^{-1}\mGm\tU^c_{11}$; hence
$VV^c\mGm^\prime=\mGm^\prime H$ or else
\[
\begin{pmatrix}
I_{\msL^\prime} & 0 \\
(VV^c)_{21} & I_{\msL^\prime}
\end{pmatrix}\begin{pmatrix}\mGm^\prime_0
\\ \mGm^\prime_1\end{pmatrix}=
\begin{pmatrix}\mGm^\prime_0H
\\ \mGm^\prime_1H\end{pmatrix}\,,\quad
H=\begin{pmatrix}
U_{11}U^*_{11} & 0 \\ 0 & U_{11}U^*_{11}\end{pmatrix}
\co\msH^\prime_\Gamma\lto\msH^\prime_\Gamma\,,
\]
\[
(VV^c)_{21}=2\img[\im B^\prime-
(K^\prime K^{-1})(\im B)(K^\prime K^{-1})^* ]\,.
\]

By $\mGm^\prime_0=\mGm^\prime_0H$ on $A^{\prime*}$,
$(I^\prime_\Gamma-H)(\mT^\prime_0)\subseteq
(I^\prime_\Gamma-H)(A^{\prime *})\subseteq\mT^\prime_0$.
Since $H^*=H$ and $\mT^{\prime*}_0=\mT^\prime_0$,
$(I^\prime_\Gamma-H)J^\prime_\Gamma(\mT^\prime_0)
\subseteq J^\prime_\Gamma(\mT^\prime_0)$;
$J^\prime_\Gamma=\bigl(\begin{smallmatrix}
0 & -\img I^\prime \\ \img I^\prime & 0
\end{smallmatrix} \bigr)$. Since
$A^{\prime *}=A^\prime\hop\mN^\prime\hop
J^\prime_\Gamma(\mN^\prime)$,
$\mN^\prime=\mT^\prime_0\cap A^{\prime\bot}$,
so $(I^\prime_\Gamma-H)J^\prime_\Gamma(\mN^\prime)
\subseteq J^\prime_\Gamma(\mT^\prime_0)\cap\mT^\prime_0=\{0\}$
\ie $\msR_{(VV^c)_{21}}=\{0\}$.
\end{proof}
\begin{rem}
In the proof the decomposition of
$A^{\prime*}$ is due to
\cite[Proposition~3.4]{Jursenas23}.
In the Krein space setting $H$ would have
$U_{11}U^c_{11}$ on the diagonal, so
$H=H^c\neq H^*$ and the invariance of
$\mT^\prime_0$ $(=\Ker\mGm^\prime_0)$ under
$I^\prime_\Gamma-H$ would not lead in general
to the invariance of
$\mT^{\prime\bot}_0=J^\prime_\Gamma(\mT^\prime_0)$
(where now
$J^\prime_\Gamma=\bigl(\begin{smallmatrix}
0 & -\img J^\prime \\ \img J^\prime & 0
\end{smallmatrix} \bigr)$)
under the same operator.
\end{rem}
\begin{cor}\label{cor:M}
Let $A$ $(A^\prime)$ be a closed simple symmetric
operator in a Hilbert space $\msH$ $(\msH^\prime)$
with equal defect numbers,
and let $(\msL,\mGm)$ $((\msL^\prime,\mGm^\prime))$
be an obt for $A^*$ $(A^{\prime *})$ with
Weyl family $M_{\mGm}$ $(M_{\mGm^\prime})$;
the Hilbert spaces
$\msL$ and $\msL^\prime$ are isomorphic.

Let $K\in\msB(\msG_0,\msL)$ bijective,
$B\in\msB(\msL)$, $C=\msB(\msG_0)$, and
let $K^\prime\in\msB(\msG_0,\msL^\prime)$ bijective,
$B^\prime\in\msB(\msL^\prime)$,
$C^\prime=\msB(\msG_0)$,
and $\im(C^\prime-C)=0$,
$\msR_{\im B}\subseteq K(\msD_{C\cap C^\prime})$.

If
\begin{equation}
C+K^*(B-M_{\mGm}(\lambda))^{-1}K=
C^\prime+K^{\prime *}(B^\prime-
M_{\mGm^\prime}(\lambda))^{-1}K^\prime
\label{eq:h}
\end{equation}
on both nonempty
$\bbC^\pm\cap\rho(A_0)\cap\rho(A^\prime_0)$,
then
$A_0=\Ker(\mGm_1-B\mGm_0)$ and
$A^\prime_0=\Ker(\mGm^\prime_1-B^\prime\mGm^\prime_0)$
are unitarily equivalent, whereas
$\Ker\mGm_0$ and
$\Ker\mGm^\prime_0$ are unitarily equivalent iff
$C=C^\prime$.
\end{cor}
\begin{proof}
As the final step
a bt
$\{\msU,\Pi\}$ with an open subset
$\msU\subseteq\bbC_*\cap \rho(A_0)\cap r(B_0)$
$(\subseteq r(A,A,A_0))$
is minimal; put $A^\sim=\mT_0$ in
Proposition~\ref{prop:CLLP}. And
$\bbC_*\cap \rho(A_0)\cap\rho(B_0)\neq\emptyset$
if $\bbC_*\cap \rho(A_0)\neq\emptyset$:
For all $\lambda\in\bbC_*\cap\rho(A_0)$,
$B^*-M_{\mGm}(\lambda)=(B-M_{\mGm}(\lambda))^*
(I_\msL-X(\lambda))$, where
$X(\lambda)=2\img(B-M_{\mGm}(\lambda))^{*-1}
\im M_{\mGm}(\lambda)\in\msB(\msL)$ is bijective;
since $X(\lambda)$ is also non-constant, there is
$\lambda$ such that either
$\norm{X(\lambda)}<1$ or $\norm{X(\lambda)^{-1}}<1$,
and in both cases $\lambda\in\rho(B_0)$.
\end{proof}
\begin{rem}
$\{\bbC_*\cap\rho(B_0),M_{\GB} \}$ is
the holomorphic element; the representing
pair $\{\Phi(\lambda),\Psi(\lambda)\}$
for the closed operator
$M_{\GB}(\lambda)=E^{-1}(M_{\GA}(\lambda))=WW^c(M_{\GA}(\lambda))$
is
$\Phi(\lambda)=I_0+L(M_{\GA}(\lambda)-C^*)$,
$\Psi(\lambda)=C\Phi(\lambda)+M_{\GA}(\lambda)-C^*$,
$L=K^{-1}(B-B^*)K^{*-1}$.
\end{rem}
\begin{rem}
If $C=C^\prime$ then \eqref{eq:h} holds iff
$\{\bbC_*\cap\rho(A_0),M_{\mathring{\Gamma}}\}
\leftrightarrow
\{\bbC_*\cap\rho(A^\prime_0),
M_{\mathring{\Gamma}^\prime}\}$---where
$M_{\mathring{\Gamma}}(\lambda)=
K^*(\re B-M_{\mGm}(\lambda))^{-1}K$
is the Weyl
function associated with
the obt $(\msL,\mathring{\Gamma})$ for $A^*$, with
$\mathring{\Gamma}_0=K^{-1}(\mGm_1-(\re B)\mGm_0)$,
$\mathring{\Gamma}_1=
-K^*\mGm_0$,
\cf \cite[Lemma~3.10]{Hassi13}---so
Theorem~\ref{thm:obt} applies after
one proves the third condition in \eqref{eq:sysV}.
In \cite[Theorem~3.11]{Hassi13}
the corollary is originally shown under hypotheses that
\eqref{eq:h} holds for $C=0=C^\prime$ and
either
a)
for all $\lambda\in\bbC^\pm\cap\rho(A_0)\cap\rho(A^\prime_0)$
or
b)
for all $\lambda\in \bbC^+\cap\rho(A_0)\cap\rho(A^\prime_0)$
if
$\Ind\mT_0\neq\msH$,
$\Ind\mT^\prime_0\neq\msH^\prime$,
and the Nevanlinna
functions $M_{\mGm}$, $M_{\mGm^\prime}$
satisfy some rather involved conditions.
(By Proposition~\ref{prop:dsptT}
$\Ind\mT_0=\{0\}=\Ind\mT^\prime_0$
if $A\in(P)$ in $\msH$ and
$A^\prime\in(P)$ in $\msH^\prime$.)
\end{rem}
\begin{rem}
Unlike
$\bbR\cap r(A_0)\cap r(B_0)=\bbR\cap\rho(A_0)$,
$\bbC_*\cap r(A_0)\cap r(B_0)$ is in general neither
larger nor smaller than $\bbC_*\cap\rho(A_0)$;
in Corollary~\ref{cor:Mo}
with $\kappa=0$ the subsets even do not intersect,
$\bbC_*\cap r(A_0)\cap r(B_0)=\bbC^-$,
$\bbC_*\cap\rho(A_0)=\bbC^+$, which is in
contrast to $A_0$ ($B_0$) parametrized by $B\in\msB(\msL)$
($B^*\in\msB(\msL)$)
in Corollary~\ref{cor:M}.
Another example of both nonempty subsets is as follows:
For an arbitrary closed relation $A_0$
in a Krein space,
if $\bbC_*\cap\sigma(A_0)$
is nowhere dense (in $\bbC$) then
$\bbC_*\cap r(A_0)\cap r(A^c_0)$
(as well as $\bbC_*\cap\rho(A_0)$) is
nonempty.
\begin{proof}
$\bbC_*\cap\rho(A_0)$ clearly is nonempty.
Assume by contradiction
$\bbC_*\cap r(A_0)\cap r(A^c_0)=\emptyset$, that is
$\bbC_*\cap r(A_0)\subseteq\bbC\setm r(A^c_0)$.
Then $\bbC_*\cap \rho(A_0)\subseteq\sigma(A^c_0)$
so
$\sigma^*(A_0)=(\bbC_*\cap\rho(A_0))\sqcup
(\sigma(A_0)\cap\sigma^*(A_0))$. Because
$\sigma(A_0)$ is a closed set while
$\bbC_*\cap\rho(A_0)\equiv\bbC_*\setm\sigma_0$
is open,
$\sigma_0\cap\ol{\bbC_*\setm\sigma_0}\subseteq
\sigma^*(A_0)$ or equivalently
$\sigma_0\cap\partial\sigma_0\subseteq\sigma^*_0$;
$\partial\sigma_0$ is the boundary of $\sigma_0$.
If by hypothesis the closure $\bar{\sigma}_0=\partial\sigma_0$,
then $\sigma_0\subseteq\sigma^*_0$
or equivalently
$\bbC_*\cap\rho(A_0)=\emptyset$.
\end{proof}
\end{rem}



\end{document}